\documentclass[12pt]{article}

\usepackage{mathtools,amsthm,amssymb,amsmath}
\usepackage{graphicx}
\usepackage{hyperref}
\usepackage{fullpage}
\usepackage{comment}

\theoremstyle{plain}
\newtheorem{theorem}{Theorem}

\newtheorem{lemma}[theorem]{Lemma}

\theoremstyle{definition}

\theoremstyle{remark}
\newtheorem{remark}[theorem]{Remark}

\def\N{\mathbb{N}}

\def\Z{\mathbb{Z}}

\def\F{\mathbb{F}}

\author{Luis H. Gallardo \\
Univ. Brest, \\
UMR CNRS 6205, \\
Laboratoire de Math\'ematiques de Bretagne Atlantique,\\
6, Av. Le Gorgeu, C.S. 93837, Cedex 3, F-29238 Brest, \\
France \\
{\tt Luis.Gallardo@univ-brest.fr}\\
AMS 2010: Primary 11T55, 11T06\\
keywords: Cyclotomic polynomials, characteristic $2$,\\
Mersenne polynomials, Fixed points of $\sigma$, Factorization.\\}

\title{Fixed points of the sum of divisors function on $\F_2[x]$}

\begin{document}

\maketitle

\abstract{We work an analogue of a classical arithmetic problem over polynomials. More precisely,
we study the fixed points $F$ of the sum of divisors function $\sigma : \F_2[x] \mapsto \F_2[x]$
(defined \emph{mutatis mutandi} like the usual sum of divisors over the integers)
 of the form $F := A^2 \cdot S$, $S$ square-free, with $\omega(S) \leq 3$, coprime with $A$, for $A$ even, of whatever degree, under some conditions. This gives a characterization of $5$ of the $11$ known fixed points of $\sigma$ in $\F_2[x]$.}

\maketitle

\section{Introduction}

We have all hear somewhere in our career that there are few positive integers $n$ with the property that the sum of all positive divisors of $n$ is a multiple of $n$. Let write the sum as 
$\sigma(n)$.  Our claim becomes then the following.  There are few solutions $n$ of the following equation.
\begin{equation}
\frac{\sigma(n)}{n} \in \N
\label{multiPZ}
\end{equation}
For example, when $n \in \{6,120\}$ we have $\frac{\sigma(6)}{6} = 2$ and $\frac{\sigma(120)}{120} = 3$. In fact this happens since we have $divisors(6) = \{1,2,3,6\}$ so that $\sigma(6) = 1+2+3+6 = 12$, and 
$$divisors(120) = \{1,2,3,4,5,6,8,10,12,15,20,24,30,40,60,120\}$$ so that $\sigma(120)=1+2+3+4+5+6+8+10+12+15+20+24+30+40+60+120 = 360$. Already here we see that we can compute $\sigma(120)$ more efficiently as follows:
Since $120 = 2^3 \cdot 3 \cdot 5$ and $\sigma(x \cdot y) = \sigma(x) \cdot \sigma(y)$ provided that $x,y$ has no common factors, we can compute:
$$
\sigma(120) = \sigma(8)\cdot\sigma(3)\cdot\sigma(5) = (1+2+4+8)\cdot(1+3)\cdot(1+5)=360.
$$
In a nutshell, in the present paper we study some arithmetic properties of an analogue to the function $n \mapsto \sigma(n)$, in which we replace $n$ by a polynomial $A(x)$ with coefficients
$0$ and $1$ only, and compute with $0,1$ as usual, besides the rule $1+1=0$ that replaces the usual rule $1+1=2$. The field $\F_2 =\{0,1\}$ in which we compute the coefficients of $A(x)$ is the simplest of all finite fields. 

For readers less familiar with finite fields, we recommend to look first at section
\ref{ReviewerA} for a simple computation with binary polynomials. Then, to look at subsections \ref{choose1}, and \ref{choose2} below. And, finally, come back to look at the rest of this Introduction.

 For all readers, we added some information about our choice of the finite field $\F_2$ for the coefficients of our polynomials
(see subsections \ref{choose1}, and \ref{choose2}) at the end of this Introduction. We also added a few comments about the role played by some small degree irreducible binary polynomials as prime factors of our perfect polynomials. This comes from an observation of one of the referees.

The paper being a little technical, we hope the following considerations will be helpful for the 
reader.

We now introduce some definitions and notation to explain the original arithmetic problem over the integers that motivated the study of our variant over the binary polynomials in $\F_2[x]$, and the link between them as well.

Let $A \in \F_2[x]$ be an irreducible polynomial, then we say that $A$ is \emph{prime}. A polynomial $M \in\F_2[x]$ is \emph{Mersenne}
(an analogue of a Mersenne number: $2^n-1$) if $M+1$ is a product of powers of $x$ and powers of $x+1$. We say that $M+1$ \emph{splits}. When a Mersenne polynomial $M$ is irreducible, we say that $M$ is a \emph{Mersenne prime}. Given a binary polynomial $B$, a binary polynomial $A$ in the sub-ring $\F_2[B]$ of $\F_2[x]$ is \emph{complete in $B$} \cite{Canaday}, if all coefficients of $A$ are equal to $1$; when $B=x$, we say simply that $A$ is \emph{complete}. A binary polynomial $B$ is \emph{odd} if $B(0)=B(1)=1$, otherwise $B$ is \emph{even}. More standard notation follows. We let $\omega(P)$ denote the number of pairwise distinct prime factors of $P \in \F_q[x]$. Likewise, we let $v_{P}(A)$ denote the valuation of the prime $P$ in the binary polynomial $A$, i.e., the least positive integer $m$, such that $P^m \mid A$ 
but $P^{m+1}\nmid A$, we also write this as $P^m \vert\vert A$. Finally, we let $\overline{\F_2}$ denote a fixed algebraic closure of $\F_2$.

We recall that a binary \emph{perfect} polynomial $A$ (see \cite{Canaday, Po, Gall-Rahav2007, Gall-Rahav2009, Gall-Rahav2012, Gall-Rahav2016, Gall-Rahav2019, Gall-RahavX, Gall-Rahav2020}) is defined by the equality
$\sigma(A)=A$, where $\sigma(A) = \sum_{D \mid A} D \in \F_2[x]$ is the sum of all divisors of $A$, including $1$ and $A$. For coprime binary polynomials $X,Y$ one has, as over the integers $\Z$, $\sigma(X Y) =\sigma(X)\sigma(Y)$. The $\sigma$ function, that maps polynomials into polynomials, is more complex than the usual sum of divisors function $\sigma_1 \colon \F_2[x] \mapsto \N$ given by $\sigma_1(A) = \sum_{D \mid A} 2^{\deg(A)}$.
For instance, some divisors $D$ of $A$ can sum up to $0$, while always a sum over $D$ of 
$2^{\deg{D}}$ is $>0$. 

It is easy to check that $0$ and $1$ are perfect polynomials, and that for any non-negative integer $n$, the polynomial $T(n) =(x(x+1))^{2^n-1}$ is (\emph{trivial}) perfect. There are only $11$ non-trivial (known) binary perfect polynomials (\emph{sporadic}), and all of them are even (see list in Lemma \ref{lesT}). Some recent computations \cite{Po}, show that new sporadic perfects must have degree exceeding $200$.

Coming back to the integers, we observe that the binary perfect polynomials are a polynomial analogue of the multiperfect numbers over $\Z$. A multiperfect number is a positive integer $n$ such that
\begin{equation}
\label{multip}
\sigma(n)/n \in \Z.
\end{equation}
Of course, we know very few about these numbers.
One see, by easy degree considerations, that for $A \in \F_2[x]$, 
\begin{equation}
\label{perfF2}
\sigma(A)/A \in \F_2[x]
\end{equation}
 is equivalent to $A = \sigma(A)$. Thus, this explains our interest in the \emph{fixed} points of $\sigma$ on $\F_2[x]$.

Technically, observe that the following problem has attracted some interest (see \cite{Agou1, Agou2, Ahmadi, Brochero, Buttler, CohenST, CohenST1, KyuMels, KyuKyu, Long, Panario, Petersson, LReis}). Given an irreducible polynomial $f$ over a finite field $\F_q$, given a polynomial $g(x)$ over the same field. How to describe the \emph{prime} (irreducible) factors of $f(g(x))$.
  
We contribute (in a special case) to this problem in the present paper, since our study of the fixed points of $\sigma$ implies that
some relations exist between the prime factors $P$ of the square-free polynomial $S$ in Lemma \ref{les6} and
the prime factors $\Phi_2(P) = 1+P$ of $\sigma(S)$. Namely, we have
\begin{equation}
\label{Panar0}
A = \sigma(A),
\end{equation}
in which we take $A$ of a special form:
\begin{equation}
\label{Panar1}
A = B^2 \cdot S = B^2 \cdot \prod_{j=1}^r P_j = \sigma(A) =\sigma(B^2) \cdot \prod_{j=1}^r (1+P_j).
\end{equation}
Therefore, equation \eqref{Panar1} gives some information about the prime factors of  $\Phi_2(P) = 1+P$ when $P$ is an odd prime divisor of $S$. See \cite{Gallardo2022} for related results obtained using the cyclotomic polynomial $\Phi_3(P) = 1+P+P^2$.

More generally, solving equation \eqref{Panar0} is a non-trivial problem of polynomial factorization in $\F_2[x]$.  See Lidl, Niederreiter \cite{Rudolf}, and Swan \cite{Swan} for known results about this problem.\\ 

The contribution of the present paper consists of giving a simple generalization of some properties of five of these $11$ known sporadic perfect polynomials.  These polynomials share a special property not shared by the other six sporadic perfect polynomials. More precisely, (see Lemma \ref{les6}), we
characterize these $5$ sporadic perfect $A$ from some special properties of their factorization
$A = B^2 \cdot S$, with $B$ even, and $S$ square-free, coprime with $B$.

Observe that we do \emph{not} fix a bound on $\omega(B)$
(so that potentially we consider many possible new even perfects (if any exists) $A$ of degree $\geq 200$ (see again \cite{Po})),
nor on the degrees of prime factors $P$ of $S$. Moreover, $P$ is not necessarily Mersenne (as was considered, e.g., in \cite{Gall-Rahav2012, Gall-Rahav2016, Gall-Rahav2019}).  Thus, we are discarding in Theorem \ref{threeC} much more  \emph{non}-perfect polynomials than in previous work (without a single computer computation).

Throughout the paper, the $1941$ work of Canaday \cite{Canaday} (see Lemma \ref{canad} and Remark \ref{core}), is important.\\

Our main result is as follows:

\begin{theorem}
\label{threeC}
Let $B \in \F_2[x]$ be an even polynomial. Assume that $\gcd(B^2,\sigma(B^2))=1$.  Let $A := B^2 P_1 \cdots P_r$, with $r \geq 1$ pairwise distinct odd prime $P_j$ such that $P_j \nmid B$.  Assume that $r \leq 3$. Then $A$ perfect implies that
\begin{equation}
\label{nwperf}
A \in \{M_{5a},M_{5b},M_{16},M_{20a},M_{20b}\},
\end{equation}
where 
\begin{equation*}
M_{5a} := x(x+1)^2 (x^2+x+1), M_{5b} := M_{5a}(x+1),
\end{equation*}
\begin{equation*}
M_{16} := x^4 (x+1)^4 (x^4+x^3+1)(x^4+x^3+x^2+x+1),
\end{equation*}
and
\begin{equation*}
M_{20a} := x^4 (x+1)^6 (x^3+x+1)(x^3+x^2+1)(x^4+x^3+x^2+x+1), 
\end{equation*}
\begin{equation*}
M_{20b} := M_{20a}(x+1). 
\end{equation*}
\end{theorem}

\begin{remark}
\label{gcdB}
For all five perfect polynomials considered in the theorem, one has the following two conditions.
\begin{equation}
\label{condRC0}
B\; \text{is even},
\end{equation}
and 
\begin{equation}
\label{condRC}
\gcd(B^2,\sigma(B^2)) = 1.
\end{equation}
\end{remark}
Moreover, observe the following.
\begin{remark}
\label{B2}
An even polynomial square $B^2$ cannot be perfect \cite[Theorem 14]{Canaday} so that $B^2 \neq \sigma(B^2)$. This also follows from Lemma \ref{canad}(a), since  $\sigma(B^2)$ is odd. In Theorem \ref{threeC} we need the stronger condition \eqref{condRC} on $B$.
\end{remark}

Furthermore, consider the following two remarks.
\begin{remark}
\label{B2RC}
By computations, it seems that for each degree $d$ there are many polynomials $B$ of degree $d$ that satisfy conditions \eqref{condRC0}, and \eqref{condRC}. More precisely, a quick computation of all even polynomials $B$ up to degree $21$ shows that
more than $68$ percent of them do satisfy \eqref{condRC}. Thus, our result applies to many polynomials $A$, as in the statement of the theorem.
Therefore, our result cover many new cases, in which we do not know if the polynomial $A$ of the theorem is perfect or not, without checking with the computer all the possible primes $P_j$ that could divide $A$.
Unfortunately, we do not see how to use our result, or our proof of the result, to obtain new even perfect polynomials (if they exist) by computations.

As one of the referees, we believe that conditions \eqref{condRC0}, and  \eqref{condRC} are so strong that it should imply, regardless of the value of $r$, the following. If $B$ satisfies the conditions, then $A = B^2 P_1 \cdots P_r$ should be one of the $5$ sporadic polynomials in the conclusion of Theorem \ref{threeC}. This, if true, seems to be a non-trivial fact. We were just able to prove it under the conditions of our theorem.
\end{remark}

\begin{remark}
\label{core}
For being able to get some progress on the remaining cases not considered in the theorem (i.e., the cases in which $r>3$) it should be necessary to generalize the results
of Canaday in Lemma \ref{canad}. This alone is a non-trivial task. Moreover, even if this task could be done, we will not be able to deduce anything about a characterization of the six other known sporadic perfects. The reason is that these $6$ polynomials are \emph{not} of the form $B^2 P_1 \cdots P_r$ (see Lemma \ref{les6}). Moreover, the $6$ remaining known sporadic perfects do \emph{not} seem to share some other interesting common property.
In other words, the more general problem to characterize \emph{all} $11$ sporadic perfects is highly non-trivial. After several years of work, we have (with Rahavandrainy) \cite{Gall-Rahav2009A, Gall-Rahav2012, Gall-Rahav2016, Gall-Rahav2019,Gall-RahavX}, merely obtained a	characterization of all $11$ sporadic perfect in a \emph{very particular} case. Namely, in the case in which every odd prime divisor $P_j$ of an even perfect polynomial $A$, is of the special form
$$ P_j = x^{a_j}(x+1)^{b_j}+1$$ for some coprime exponents $a_j,b_j$ (i.e., each $P_j$ is a Mersenne polynomial). Of course, prime divisors of $A$ need \emph{not} be Mersenne polynomials.

Theorem \ref{threeC} is a first (modest) step to study the new case in which we assume that the prime divisors $P_j$ of an even perfect polynomial $A$ are \emph{not} necessarily Mersenne  polynomials.\\

Now, let us come back to the case $r>3$ of our approach. We know that this approach works to characterize the $5$ known sporadic perfects of the form $B^2 P_1 \cdots P_r$. But fails to characterize all known sporadic perfects.

However, me may add the following.
Essentially, (in the proof of the theorem) we use properties of the prime factors of general (not necessarily prime) Mersenne  polynomials $M$, i.e., polynomials with the property that $M+1$ has all its roots in $\F_2$. Now consider binary polynomials $M_g$, with the property that all roots of $M_g+1$
belong to an appropriate non-trivial extension field of $\F_2$ (e.g., belong to $\F_4$). We believe that understanding the factorization of these \emph{general Mersenne} polynomials $M_g$ can help to get some progress in the case when $r >3$.  However, even a simple preliminary study of this special case, appears to be a difficult non-trivial problem.
\end{remark}
Finally, we discuss the following two matters suggested by a referee.

\subsection{Choice of $\F_2$ as ground field for the coefficients of our polynomials}
\label{choose1}
The first reason for the choice is that the ring $\F_2[x]$ is considered as the closest analogue to the ring of integers $\Z$  to work arithmetic problems.

 The second (and more important) reason  for the choice is  the following.
We have no analogue of Canaday's results \cite{Canaday} over $\F_2[x]$ for other rings $\F_p[x]$, for $p$ an odd prime, nor for more general rings $\F_q[x]$ with $q$ a power of a prime. One reason for this is that the general problem of factorization into irreducible polynomials is much more complex  when the characteristic of the ring is $>2$.
This happens, regardless of the existence of many papers on the subject (see \cite{Beard77Fq,Beard77FqA,Beard78Fq,Beard91Fp,Beard97Fq,Gall-Rahav2005F4,Gall-Rahav2007F4,Gall-Rahav2008F3,Gall-Rahav2009Fp2,Gall-Rahav2009F4,Gall-Rahav2010Fp2,Gall-Rahav2011,Gall-Rahav2011Fpp,Gall-Rahav2012Fp,Gall-Rahav2014,Gall-Rahav2016Fp}).

\subsection{Role of small degree prime factors of even perfect polynomials in the present paper}
\label{choose2}

First, observe that the irreducible polynomials of degree $5$ or more of $\F_2[x]$ do not play any role in the paper. For which reason? The simple reason is that the only known perfect polynomials over $\F_2$ are all even and have irreducible factors of degrees $1,2,3,4$ only (see Lemma \ref{lesT}). Of course, it may exist unknown binary perfect polynomials $A$ with irreducible factors of any degree, but none such $A$ is known with degree $\leq 200$ (see \cite{Po}). Moreover, $\omega(A) \geq 5$ (see \cite{Gall-Rahav2009,Gall-Rahav2009A}). Furthermore, the main results used in the proof, namely the results in Lemma \ref{canad}, have the following property.
They reduce the study of irreducible factors of an even perfect polynomial of any degree to the study of small degree irreducible factors that all have degree less than $5$.

Even perfect polynomials, by definition, should have at least one linear factor. Indeed, they are divisible by both linear factors $x$ and $x+1$. In particular, if they are divisible only by $2$ irreducible factors they must be a product of a power of $x$ by a power of $x+1$. It is easy to prove that in fact the exponents must be equal, and of the form $2^n-1$. Thus, these polynomials coincide with the trivial perfects $T(n)$ (see also Section \ref{ReviewerA}).

The linear factors $x,x+1$ appear everywhere in the proof of the theorem. The reason is the following. For each odd irreducible factor $P$ that divides exactly a binary even perfect $A$ (i.e., such that $P$ divides $A$ but $P^2$ do not divide $A$) we have that $\sigma(P)=P+1$ divides also $\sigma(A)=A$. Thus, by definition of odd polynomial it is easy to see that $P+1$ is even, so that $x(x+1)$ divides $P+1$.

\section{A simple computation with binary polynomials}
\label{ReviewerA}

We will work with polynomials over the  smallest finite field. Namely, $\F_2 =\{0,1\}$.
First, let us observe that since the list of all divisors of $x$ is $[x,1]$, one has $\sigma(x) =x+1$.
By translation $x \mapsto x+1$, we deduce that $\sigma(x+1) =(x+1)+1 = x+(1+1) = x+0 = x$.
Now, the property, $\sigma(AB) = \sigma(A)\sigma(B)$,  provided that $A,B$ are coprime, implies that
\begin{equation}
\label{perf1}
\sigma(x(x+1))= \sigma(x) \sigma(x+1) =(x+1)x = x(x+1).
\end{equation}
We have then found the perfect polynomial with the smallest degree $>0$, namely $T(1)=x(x+1)$.
We can write $T(1)$ as follows: $T(1) = x^{2^1-1}(x+1)^{2^1-1}$.
Following the same lines of computation, one proves easily by induction that if 
$T(n) = x^{2^n-1}(x+1)^{2^n-1}$ is perfect, then the same holds for $T(n+1)$. 

Thus, we have infinitely many even perfect polynomials (that we call \emph{trivial} perfect).
Unfortunately, we cannot obtain more perfect polynomials with similar methods.
The list of all known perfect polynomials (see Lemma \ref{lesT}) was obtained by computer computations.
We believe that this list cover \emph{all} perfect polynomials. However, we are very far to build  a proof (or a disproof) of this.
The present paper explores a small part of this problem, using elementary methods, like the preceding computation.
We have no choice, there is \emph{no} (known) more sophisticated methods to treat this problem.

\section{Tools}
\label{toool}

The following lemma contains a simple (new) observation in part (a), and summarizes some useful results of Canaday \cite{Canaday} in parts (b) to (f).

\begin{lemma}
\label{canad}
\begin{itemize}
\item[\rm{(a)}]
Let $P$ be prime, and let $n$ be a positive integer. Then $\sigma(P^{2n})$ is odd.
In particular, $\sigma(C^2)$ is odd, for any binary polynomial $C$.
\item[\rm{(b)}]
If $A = x^{h-1}+x^{h-2}+\cdots+1$ is a complete polynomial and $(x+1)^r$ divides $A$ but $(x+1)^{r+1}$ does not, then $r=2^n-1$ and $A = (x+1)^{2^n-1}B^{2^n}$ where $B$ is complete.
\item[\rm{(c)}]
The only complete and irreducible polynomials of the form $x(x+1)^{\beta}+1$ are $x^2+x+1$ and $x^4+x^3+x^2+x+1$.
\item[\rm{(d)}]
The only complete $A =x^{2m}+\cdots+1$ whose irreducible factors are of the form $x^{\alpha}(x+1)^{\beta}+1$ are $x^2+x+1,x^4+x^3+x^2+x+1, (x^3+x+1)(x^3+x^2+1)$.
\item[\rm{(e)}]
It is impossible to have $\sigma(x^{2k}) = \sigma(P^2)$ or, more generally, $\sigma(Q^{2m})=\sigma(P^{2n})$ for irreducible polynomials $P,Q \in \F_2[x]$.
\item[\rm{(f)}]
The polynomial $P=x(x+1)^{2^m-1}+1$ is irreducible only for $m=1$ and $m=2$.
\end{itemize}
\end{lemma}

\begin{proof}
We prove (a). One sees that $S :=\sigma(P^{2n})$ is a sum of $2n+1$ nonzero monomials $P^k$.
If $\deg(P) >1$, we have $P(0)=P(1)=1$ since $P$ is prime, thus $S(0)=S(1)=1$. If $P=x$ then
$S(0)=P(0)=1$, $S(1)=2n+1=1$ in $\F_2$. Similarly, if $P=x+1$ then $S(1)=P(1)=1$, and
$S(0)=2n+1=1$ in $\F_2$.  Put $C = \prod_{j} P_j^{n_j}$, for some primes $P_j$, thus $\sigma(C^2)= \prod_{j}\sigma(P_j^{2 n_j})$ is odd as product of odd polynomials.

Part (b) is \cite[Lemma 1]{Canaday}. Part (c) is \cite[Corollary]{Canaday}. Part (d) is \cite[Theorem 8]{Canaday}. Likewise, part (e) is \cite[Lemma 14]{Canaday}, and part (f) is  \cite[Lemma 2]{Canaday}.
\end{proof}

The list of all known \cite{Canaday} sporadic perfect follows. Gallardo and Rahavandrainy \cite{Gall-Rahav2009, Gall-Rahav2009A} proved
that the list contains all the sporadic perfects $M$ with $\omega(M) \leq 4$. The case $\omega(M) = 5$ is open from $2009$.

\begin{lemma}
\label{lesT}
With the primes
\begin{equation*} 
Q_2 := x^2+x+1,\; Q_{3a} := x^3+x+1,\; Q_{3b} := x^3+x^2+1,\;
Q_{4a} :=x^4+x^3+1,
\end{equation*}
\begin{equation*} 
 Q_{4b} := x^4+x^3+x^2+x+1, Q_{4c} := x^4+x+1;
\end{equation*} 
one has the $11$ sporadic perfects known. Besides,  $M_{20a}$ and $M_{20b}$, they are the unique sporadic perfects with at most four
distinct prime divisors.
\begin{equation*}
M_{5a} := x(x+1)^2 \cdot Q_2,
M_{5b} := (x+1) x^2 \cdot Q_2,
M_{11a} := x(x+1)^2 \cdot Q_2^2 \cdot Q_{4c},
\end{equation*} 
\begin{equation*}
M_{11b} := x^2 (x+1)\cdot Q_2^2 \cdot Q_{4c},
M_{11c} := x^3 (x+1)^4 \cdot Q_{4a},
M_{11d} := x^4 (x+1)^3 \cdot Q_{4b},
\end{equation*} 
\begin{equation*}
M_{15a} := x^3 (x+1)^6 \cdot Q_{3a} \cdot Q_{3b},
M_{15b} := x^6 (x+1)^3 \cdot Q_{3a} \cdot Q_{3b},
M_{16} := x^4 (x+1)^4 \cdot Q_{4a} \cdot Q_{4b},
\end{equation*} 
\begin{equation*}
M_{20a} := x^4 (x+1)^6 \cdot Q_{3a}\cdot Q_{3b} \cdot Q_{4b},
M_{20b} := x^6 (x+1)^4 \cdot Q_{3a} \cdot Q_{3b} \cdot Q_{4a}.
\end{equation*}
\end{lemma}

With the same notations of Lemma \ref{lesT}, the list of the five sporadic perfects of a special form follows.

\begin{lemma}
\label{les6}
Besides $M_{20a}$ and $M_{20b}$ the following polynomials $A$ are the only  sporadic perfects with $\omega(A) \leq 4$, of the form
\begin{equation}
\label{A2S}
A := B^2 \cdot S,
\end{equation}
where $B$ is the even polynomial of higher degree, such that $B^2 \vert A$, and $S$
is a square-free polynomial coprime with $B$, i.e., one has $\gcd(B,S)=1$.
\begin{equation*}
M_{5a} = (x+1)^2 \cdot x \cdot Q_{2},
M_{5b} = x^2 \cdot (x+1) \cdot Q_{2},
M_{11a} = {((x+1)Q_2)}^2 \cdot x \cdot Q_{4c},
\end{equation*} 
\begin{equation*}
M_{11b} = {(x Q_2)}^2 \cdot (x+1) \cdot Q_{4c},
M_{16} = (x^2(x+1)^2)^2 \cdot Q_{4a} \cdot Q_{4b},
\end{equation*} 
\end{lemma}

We easily check the following lemma. It is useful for the proof  of  the last part of the theorem. 

\begin{lemma}
\label{identity2th}
Let  $a = 2^n k$ be an even number, where $k$ is odd. For any binary polynomial $A$, and positive integer $r$, set $S(A^{r}) := 1+A+\cdots+A^{r}$. Then
\begin{equation*}
S(A^{a})+1 = A \cdot (A+1)^{2^n-1} \cdot S(A^{k-1})^{2^n}.
\end{equation*}
\end{lemma}

\section{Proof of Theorem \ref{threeC}}

Remember that $r$ is the number of odd prime divisors of the even perfect polynomial $A$.
We consider the cases $r=1$, $r=2$, and $r=3$. In each of them we will work on the equality
$$
A =\sigma(A),
$$
with both $A$ and $\sigma(A)$ explicitly factored as product of primes in $\F_2[x]$.
We apply our lemmas in the section \ref{toool} to prove the result in each of these cases.
Essentially, our method consists of using the uniqueness of the factorization into primes in the ring $\F_2[x]$.\\

We assume that $r=1$. Thus, for some prime $P_1$ one has
\begin{equation}
\label{star1}
\sigma(B^2)(P_1+1) = B^2 P_1.
\end{equation}
Since $\gcd(B^2,\sigma(B^2))=1$ and $P_1$ is prime, \eqref{star1} implies that $\sigma(B^2)=P_1$. Thus, $P_1 = (1+B)^2$. This is impossible. Therefore, this case does not happen.\\

We assume that $r=2$. For some primes $P_1,P_2$ we have
\begin{equation}
\label{star2}
\sigma(B^2)(P_1+1)(P_2+1) = B^2 P_1 P_2.
\end{equation}

Equation \eqref{star2} can also be written as
\begin{equation}
\label{star3}
P_1 P_2(B^2+\sigma(B^2) = (P_1+P_2+1) \sigma(B^2).
\end{equation}

Since $\gcd(\sigma(B^2),B^2)=1$, \eqref{star2} implies that  $\sigma(B^2) \mid P_1P_2$.

Case 1.
We can assume that $\sigma(B^2) = P_1$. Thus,  $\omega(B^2) = 1$. Therefore,
$ A = B^2 P_1P_2$ is an even perfect polynomial with $\omega(A)=3$. This implies that $A \in \{M_{5a},M_{5b}\}$, by Lemma \ref{les6} and Lemma \ref{lesT}.

Case 2. We have then
\begin{equation}
\label{star4}
\sigma(B^2) = P_1P_2.
\end{equation}
Since $B^2$ is an even square, \eqref{star4} together with Lemma \ref{canad} (a), imply that both $P_1$ and $P_2$ are odd. As before, \eqref{star4} implies that $\omega(B^2) \leq 2$, so that
$A$ is an even perfect polynomial with $\omega(A) \leq 4$. By Lemma \ref{les6} and Lemma \ref{canad}, the only possibility is $A = M_{16}$, for which $B = x^2(x+1)^2, P_1 = x^4+x^3+x^2+x+1, P_2 = x^4+x^3+1$.\\

We assume now that $r=3$. We have then
\begin{equation}
\label{star5}
\sigma(B^2)(P_1+1)(P_2+1)(P_3+1) = B^2 P_1P_2 P_3.
\end{equation}

Case 1.
We have $\omega(\sigma(B^2)) =1$, say $\sigma(B^2) = P_1$. Thus, as before,  $\omega(B^2) = 1$. This implies that $\omega(A)=4$. By Lemma \ref{les6}, this case does not happen.

Case 2.
We have $\omega(\sigma(B^2)) =2$. If $\omega(B) =1$, as before, there is no solution by Lemma \ref{les6}. We assume then that $\omega(B)=2$. One sees that $\omega(A)=5$ now, thus we cannot deduce the result from Lemma \ref{les6} again.  In fact, we do not know if  $M_{20a}$ and $M_{20b}$ are the unique even perfects $M$ with $\omega(M) =5$. 

We have, by Lemma \ref{canad}(a), and without loss of generality,  that for odd primes $P_1,P_2$, for primes $R_1 \neq R_2$, and for positive integers $a_1,a_2$ the following holds.
\begin{equation}
\label{star6}
\sigma(B^2)= P_1 P_2\;\text{, and}\;B = R_1^{a_1}R_2^{a_2}.
\end{equation}
Moreover, \eqref{star5} becomes
\begin{equation}
\label{star7}
(P_1+1)(P_2+1)(P_3+1) = B^2 P_3.
\end{equation}
Assume that $P_3$ is even. If $P_3=x$, 
since $\gcd(P_3,B)=1$, and $B$ is even, we have that $R_1 =x+1$, and $R_2$ is odd. Moreover,
$P_1$ and $P_2$ are odd, hence comparing valuations in \eqref{star7} gives
$v_{x}((P_1+1)(P_2+1)(P_3+1)) \geq 2$, while $v_{x}(B P_3) =1$.  Thus, $P_3 \neq x$. By translation, $x$ to $x+1$, $P_3 \neq x+1$. Therefore, $\deg(P_3) >1$. Since $P_1,P_2,P_3$ are all odd, it follows from \eqref{star7} that, say, $R_1 =x$ and
$R_2 =x+1$, $B$ is even, $\gcd(B,P_3)=1$, and $\omega(B)=2$.
It follows from \eqref{star6} that we can take $\sigma(x^{2a_1}) = P_1$ and $\sigma((x+1)^{2a_2}) = P_2$, so that
\begin{equation}
\label{star8}
P_1+1 = x(1+x+ \cdots +x^{2a_{1}-1}),
\end{equation}
and
\begin{equation}
\label{star9}
P_2+1 = x(1+x+ \cdots +x^{2a_{2}-1}).
\end{equation}
From \eqref{star8} and \eqref{star9} we get $v_x(P_1+1)=v_x(P_2+1)=1$.
Since $x,x+1$ and $P_3$ are the only primes that divide $B^2P_3$, we can assume that, say,
$P_3 \mid P_1+1$ and $P_3 \nmid P_2+1$.  Write, $P_1+1 = x^{c_{1}}(x+1)^{c_{2}}P_3$,
$P_2+1 = x^{d_{1}}(x+1)^{d_{2}}$, and $P_3+1 = x^{e_{1}}(x+1)^{e_{2}}$.
From \eqref{star8} and \eqref{star9} we get $c_{1} =1$ and $d_{1} =1$. 

Since $P_{1} = 1+ x(x+1)^{c_{2}}P_3$ we have from \eqref{star8}
\begin{equation}
\label{star14}
(P_{1}+1)/x = \sigma(x^{2a_{1}-1}) =1+x+ \cdots +x^{2a_{2}-1} = (x+1)^{c_{2}}P_3.
\end{equation}
Thus, $(x+1)^{c_{2}}P_3$ is complete. It follows from Lemma \ref{canad}(b) that $c_{2} = 2^{n}-1$ for some positive integer $n$, since $P_1$ is odd. In other words, for $K$ complete,
we have the following equality:
\begin{equation}
\label{star15}
(P_{1}+1)/x = (x+1)^{2^{n}-1} K^{2^{n}}.
\end{equation}
It follows from \eqref{star15} and \eqref{star14} that $P_3 = K^{2^{n}}$. Hence, $n=0$. This is impossible.
Therefore, Case $2$ does not happen.

Case 3.
We have $\omega(\sigma(B^2)) =3$. Thus, we consider again \eqref{star5}, i.e.,
\begin{equation}
\label{star16}
\sigma(B^2)(P_1+1)(P_2+1)(P_3+1) = B^2 P_1P_2P_3.
\end{equation}
Equation \eqref{star16} implies immediately
\begin{equation}
\label{star17}
\sigma(B^2) = P_1P_2P_3,
\end{equation}
and
\begin{equation}
\label{star18}
(P_1+1)(P_2+1)(P_3+1) = B^2.
\end{equation}
Since $P_1,P_2$ and $P_3$ are all odd, \eqref{star18} implies that $x(x+1) \mid B$.
In particular, $\omega(B) \geq 2$.
Since $B^2$ is an even square, $\sigma(B^2)$ is odd, so that \eqref{star17} implies that $P_1,P_2$ and $P_3$ are all odd.
Thus, \eqref{star17} implies that $\omega(B) = \omega(B^2) <4$.

If $\omega(B)=2$, one has $B = x^{a}(x+1)^{b}$ for positive integers $a,b$.
Since $P_1P_2P_3$ is square free, we can assume from \eqref{star17} that, say
\begin{equation}
\label{star19}
\sigma(x^{2a}) = P_3\;\text{, and}\;\sigma((x+1)^{2b}) = P_1P_2.
\end{equation}
Moreover, since $P_3$ is odd, for some positive integers $c,d$ we have
\begin{equation} 
\label{star20}
1+ P_3 = x^{c}(x+1)^{d}.
\end{equation}
From \eqref{star19} and \eqref{star20} we obtain $c=1$, since $1+P_3=x(1+\cdots x^{2a-1})$.
Putting $K_3 = (1+P_3)/x$, one sees that
\begin{equation} 
\label{star21}
K_3 = 1+\cdots +x^{2a-1} = x^{c-1}(x+1)^d =(x+1)^d.
\end{equation}
Equation \eqref{star21} says that $K_3$ is complete, thus, as before, Lemma \ref{canad}(b) implies that for some positive integer $n$ one has $d =2^n-1$, and
$K_3 = (x+1)^{2^n-1}C^{2^n}$, with $C$ complete. This forces $C=1$. Hence,
\begin{equation} 
\label{star22}
P_3 = 1+x(x+1)^{2^n-1}.
\end{equation}
Since $P_3$ is prime, Lemma \ref{canad}(c) implies that
\begin{equation} 
\label{star23}
P_3 \in \{x^2+x+1,x^4+x^3+x^2+x+1\}.
\end{equation}
Assume that $P_3 =x^2+x+1$. Thus, from $\sigma(x^{2a}) = P_3$ we get $a=1$.
In other words, $B = x(x+1)^b$. From \eqref{star18} and \eqref{star19} we obtain
\begin{equation} 
\label{star24}
(P_1+1)(P_2+1) = x(x+1)^{2b-1}.
\end{equation}
Equation \eqref{star24} is impossible since $v_x((P_1+1)(P_2+1)) \geq 2$, while $v_x(x(x+1)^{2b-1}) =1$.  Thus $P_3 \neq x^2+x+1$.
Assume then that we have $P_3 = x^4+x^3+x^2+x+1$. 
We claim that $A = M_{20b}$
($M_{20a}$ is obtained by the same method, switching $x$ and $x+1$).
In order to prove the claim, observe that $P_3+1=x(x+1)^3$, thus \eqref{star18} becomes
\begin{equation} 
\label{star25}
(P_1+1)(P_2+1) = x^{2a-1}(x+1)^{2b-3}.
\end{equation}
From $\sigma(x^{2a})=P_3$ we get $a=2$. This together \eqref{star19} gives 
$B = x^2(x+1)^b$ and
\begin{equation} 
\label{star26}
(P_1+1)(P_2+1) = x^{3}(x+1)^{2b-3}.
\end{equation}
We can take in \eqref{star26}, with positive integers $b_{1},b_{2}$; $b_2$ odd since $P_2$
is not a square, and $b_1$ even since $b_1+b_2=2b-3$. Thus,
\begin{equation} 
\label{star27}
P_1+1 = x(x+1)^{b_{1}}, P_2+1 = x^2(x+1)^{b_{2}}.
\end{equation}

Since $\sigma((x+1)^{2b})$ is complete in $x+1$, and since $P_1,P_2$ are Mersenne
 one sees  that \eqref{star19} together with Lemma \ref{canad}(d) implies that
 \begin{equation} 
\label{star28}
\sigma((x+1)^{2b}) \in \{ x^2+x+1,x^4+x^3+1,(x^3+x+1)(x^3+x^2+1)\}.
\end{equation}

But $\omega(\sigma((x+1)^{2b}) =2$, since $\sigma((x+1)^{2b})  = P_1P_2$.
Thus, the only possibility allowed by \eqref{star28} is that $\sigma((x+1)^{2b}) = (x^3+x+1)(x^3+x^2+1)$. Therefore, $P_1 = x^3+x+1, P_2 = x^3+x^2+1$, i.e.,  $b=3$. Thus, $B=x^2(x+1)^3$. In other words, we have
\begin{equation} 
\label{star29}
B^2 P_1P_2P_3 = M_{20b}.
\end{equation}

This finishes the case in which $\omega(B)=2$.

We claim that the remaining case, namely $\omega(B)=3$ does not happen. To prove the claim, we assume that, on the contrary, $B = R_1^{a_1}R_2^{a_2}R_3^{a_3}$ with some positive integers $a_1,a_2,a_3$. Observe that the perfect polynomial $A =B^2 P_1P_2P_3$
 has $\omega(A)=6$ so that, as before, we cannot rely on Lemma \ref{les6} for the proof. But, we can, and do, assume that $R_1=x$, $R_2=x+1$, and that $R_3$ is odd, 
 since $x(x+1) \mid B$ (see \eqref{star18}). Thus, \eqref{star17} becomes
\begin{equation}
\label{star30}
\sigma(x^{2a_1})\sigma(x^{2a_2})\sigma(R_3^{2a_3}) = P_1P_2P_3.
\end{equation}

Since $P_1P_2P_3$ is square-free, the three factors on the left-hand side of \eqref{star30}
are pairwise coprime, so that we can take
\begin{equation} 
\label{star31}
\sigma(x^{2a_1}) = P_1,\sigma(x^{2a_2}) = P_2,\sigma(R_3^{2a_3})=P_3.
\end{equation}
Put, $2a_1 = 2^{n_1}k_1, 2a_2 = 2^{n_2}k_2, 2a_3 = 2^{n_3}k_3$, for odd numbers $k_1,k_2,k_3$.  From Lemma \ref{identity2th} we get
\begin{equation} 
\label{star32}
P_1+1 = \sigma(x^{2a_1}) +1 = x(x+1)^{2^{n_1}-1}(1+x+ \cdots + x^{k_1-1})^{2^{n_1}}.
\end{equation}
\begin{equation} 
\label{star33}
P_2+1 = \sigma(x^{2a_2}) +1 = x(x+1)^{2^{n_2}-1}(1+x+ \cdots + x^{k_2-1})^{2^{n_2}}.
\end{equation}
\begin{equation} 
\label{star34}
P_3+1 = \sigma(R_3^{2a_3}) +1 = R_3(R_3+1)^{2^{n_3}-1}(1+R_3+ \cdots + R_3^{k_3-1})^{2^{n_3}}.
\end{equation}
On the other hand, \eqref{star18} implies
\begin{equation} 
\label{star35}
P_1+1 = x^{u_1}(x+1)^{u_2}R_3^{u_3},
\end{equation}
\begin{equation} 
\label{star36}
P_2+1 = x^{v_1}(x+1)^{v_2}R_3^{v_3},
\end{equation}
and
\begin{equation} 
\label{star37}
P_3+1 = x^{w_1}(x+1)^{w_2}R_3^{w_3}.
\end{equation}

Assume, first, that $k_1=k_2=1$. Thus, from \eqref{star35} and \eqref{star36} we get
\begin{equation} 
\label{star38}
P_1+1 = x(x+1)^{2n_1-1},
\end{equation}
and
\begin{equation} 
\label{star39}
P_2+1 = (x+1) x^{2n_2-1}.
\end{equation}
But from \eqref{star34} and \eqref{star37} we have $w_3=1$.  Thus, \eqref{star38} and \eqref{star39}
implies that
\begin{equation} 
\label{star40}
v_{R_3}(P_1+1)(P_2+1)(P_3+1) =1.
\end{equation}
Clearly, \eqref{star40} contradicts \eqref{star18}. Thus, the case $k_1=k_2=1$ does not happen.\\

We claim that the case $k_1>1$ and $k_2 >1$ also does not happen.

Since from \eqref{star18} $(P_1+1)(P_2+1)(P_3+1)=B^2$ and since $\omega(B)=3$
for all $j$, one has $2 \leq \omega(P_j+1) \leq 3$. Moreover, one sees that $\omega(P_1+1)=2$
is equivalent to $k_1=1$, and $\omega(P_2+1)=2$ is equivalent to $k_2=1$. Thus, $k_1>1$ and
$k_2>1$ forces $R_3 = 1+x+ \cdots + x^{k_1-1}$, and $R_3 =1+(x+1)+ \cdots +(x+1)^{k_2-1}$. In other words, we have $\sigma(x^{k_1-1})=\sigma((x+1)^{k_2-1})$. This is impossible by Lemma \ref{canad}(e).

By the same argument, one sees that it remains only two possibilities, either Case A, or Case B:

Case A. One has $k_1=1$, $k_2>1$, and $R_3 = 1+(x+1)+ \cdots +(x+1)^{k_2-1}$.
Case B. One has $k_1>1$, $k_2=1$, and $R_3 = 1+x+\cdots+x^{k_1-1}$.

We work now Case A: We have $2a_1 =2^{n_1}$, with $n_1 \geq 1$.
We have $P_1=1+x(x+1)^{2^{n_1-1}}$. It follows from Lemma \ref{canad}(f) that $n_1 \in \{1,2\}$, i.e., that $a_1 \in \{1,2\}$. Thus, $$P_1 \in \{x^2+x+1,x^4+x^3+x^2+x+1\}.$$

Case A1. Assume that $P_1 = x^2+x+1$. Thus, $n_1=1=a_1$, so that $B=x(x+1)^{a_2}R_3^{a_3}$. Thus, \eqref{star18} becomes
\begin{equation} 
\label{star41}
(P_2+1)(P_3+1) = x(x+1)^{2a_2-1}R_3^{2a_3}.
\end{equation}

I will now recall that \eqref{star31} implies $\sigma(x^2)=P_1,\sigma((x+1)^{2a_2})=P_2$,
and $\sigma(R_3^{2a_3})=P_3$, with $2a_2 =2^{n_2}k_2, 2a_3 =2^{n_3}k_3$, $k_2>1$ is odd, and $k_3 \geq 1$ is odd.
 
But $P_2$ and $P_3$ are both odd, thus $v_{x}((P_2+1)(P_3+1)) \geq 2$, while \eqref{star41} implies that $v_{x}(x(x+1)^{2a_2-1}R_3^{2a_3}) =1$. This is impossible. Thus, Case A1 does not happen.\\

Case A2. Assume that $P_1 = x^4+x^3+x^2+x+1$. Thus, $a_1=2$, so that $B=x^2(x+1)^{a_2}R_3^{a_3}$. Thus, after division of both sides by $x(x+1)^3$, equation \eqref{star18} becomes 
\begin{equation} 
\label{star42}
x^3(x+1)^{2a_2-3}R_3^{2a_3}= (P_2+1)(P_3+1),
\end{equation}
with $a_2 \geq 2$.
Here, we have from \eqref{star31}, $\sigma(x^4)=P_1,\sigma((x+1)^{2a_2})=P_2$, and $\sigma(R_3^{2a_3})=P_3$. By \eqref{star33} and \eqref{star34} we have $v_{R_3}(P_2+1)=2^{n_2}$ and $v_{R_3}(P_3+1)=1$.
 
Thus, $v_{R_3}((P_2+1)(P_3+1)=v_{R_3}(P_2+1)+v_{R_3}(P_3+1) = 2^{n_2}+1$. On the other hand, from  \eqref{star32} we obtain $v_{R_3}(x^3(x+1)^{2a_2-3}R_3^{2a_3})=2a_3$. Thus, $2a_3 = 2^{n_2}+1$. This is impossible, thus Case A2 does not happen.

Thus, Case A does not happen.\\

Case B. We have now, $k_2=1$ and $k_1>1$.  Thus, $2a_2 =2^{n_2}$, $2a_1= 2^{n_1}k_1$
and 
$$P_1 = 1+x(x+1)^{2^{n_1}-1}(1+\cdots+x^{k_1-1})^{2^{n_1}}.$$ Since $k_2=1$ one has
\begin{equation} 
\label{star43}
P_2 = 1+x^{2^{n_2}-1}(x+1).
\end{equation}
Since $P_2$ is prime, Lemma \ref{canad}(f), \eqref{star43}, and switching
$x$ and $x+1$ gives $n_2 \in \{1,2\}$. If $n_2=1$ then $a_2=1$ so that $P_2=x^2+x+1$, while if $n_2=2$ then $a_2=2$ and $P_2 =x^4+x^3+1$.\\

Case B1. We have $P_2+1=x(x+1)$. In particular, $k_2=1$ and $n_2=1$. More precisely, we have $2a_1=2^{n_1}k_1,2a_2=2^{n_2}k_2 =2,2a_3=2^{n_3}k_3$.

We have also $$P_1+1 = x(x+1)^{2^{n_1}-1}(1+\cdots+x^{k_1-1})^{2^{n_1}},$$ and $$P_3+1 = R_3(R_3+1)^{2^{n_3}-1}(1+\cdots+R_3^{k_3-1})^{2^{n_3}}.$$

We have thus, by definition of $B$
\begin{equation} 
\label{star44}
B^2 = x^{2^{n_1}k_1}(x+1)^2R_3^{2^{n_3}k_3}.
\end{equation}
Divide now both sides of \eqref{star18} by $x(x+1)=P_2+1$ to get
\begin{equation} 
\label{star45}
(P_1+1)(P_3+1) = x^{2^{n_1}k_1}(x+1)R_3^{2^{n_3}k_3}.
\end{equation}
Since $P_3$ and $P_1$ are odd primes \eqref{star45} implies
\begin{equation} 
\label{star46}
2 \leq v_{x+1}((P_1+1)(P_3+1)) = v_{x+1}( x^{2^{n_1}k_1}(x+1)R_3^{2^{n_3}k_3})=1.
\end{equation}
Since \eqref{star46} is impossible, we obtain that Case B1 does not happen.\\

Case B2. Here, $P_2+1=x^3(x+1)$. In particular, $k_2=1$ and $n_2=2$. More precisely, we have $2a_2=2^{n_2}k_2 = 4$. As before, we have by definition of $B$
\begin{equation} 
\label{star47}
B^2 = x^{2 a_1}(x+1)^4 R_3^{2a_3}.
\end{equation}
Divide now both sides of \eqref{star18} by $x^3(x+1)=P_2+1$ to get
\begin{equation} 
\label{star48}
 x^{2 a_1-3}(x+1)^3 R_3^{2a_3} = (P_1+1)(P_3+1).
\end{equation}

We have now 
$$P_1+1 =x(x+1)^{2^{n_1}-1}(1+\cdots+x^{k_1-1})^{2^{n_1}},$$
$$P_3+1 = R_3(R_3+1)^{2^{n_3}-1}(1+\cdots+R_3^{k_3-1})^{2^{n_3}}.$$
Computing the valuation in $R_3$ in both sides of \eqref{star48} we obtain
\begin{equation} 
\label{star49}
 2a_3 = 2^{n_1}+1.
\end{equation}
Since \eqref{star49} is impossible, we obtain that Case B2 does not happen.
This finish the proof that the case $\omega(B)=3$ does not happen.
Thus, we proved the theorem.

\end{document}